\documentclass[11pt]{amsart}
\usepackage{amsmath,amssymb,amsthm,amsfonts}
\usepackage{paralist}
\usepackage{booktabs}
\usepackage[usenames,dvipsnames]{color}
\usepackage{comment}
\usepackage{graphicx,epsf}
\usepackage{graphics}
\usepackage{epsfig}
\usepackage{epstopdf}
\usepackage{psfrag}
\usepackage{latexsym}
\usepackage[latin1]{inputenc}
\usepackage[shortlabels]{enumitem}
\usepackage{thmtools}
\usepackage{url}
\usepackage[all,cmtip]{xy}
\usepackage{hyperref} 
\usepackage{pifont}

\usepackage[enableskew]{youngtab}
\usepackage{ytableau}

\newcommand{\LR}{\rLR}

\title[Bounds on the largest Kronecker multiplicities]
{Bounds on the largest Kronecker and induced \\ multiplicities of finite groups}

\author[Igor Pak, Greta Panova and Damir Yeliussizov]{
Igor Pak$^\star$, \ \  Greta Panova$^\dagger$ \ \ and \ \ Damir Yeliussizov$^\star$}

\thanks{\today}
\thanks{\thinspace ${\hspace{-.45ex}}^\star$Department of Mathematics,
UCLA, Los Angeles, CA~90095.
\hskip.06cm
Email:
\hskip.06cm
\texttt{{pak,\ts damir}@math.ucla.edu}}

\thanks{\thinspace ${\hspace{-.45ex}}^\dagger$Department of Mathematics,
 UPenn, Philadelphia, PA~19104 and IAS, Princeton, NJ~08540.
\hskip.06cm
Email:
\hskip.06cm
\texttt{panova@math.upenn.edu}}

\newcommand{\Ind}{\mathrm{Ind}}

\newcommand{\upa}{\uparrow}
\newcommand{\doa}{\downarrow}

\newcommand{\Irr}{{\operatorname{Irr} } }
\newcommand{\Conj}{\operatorname{Conj } }

\newtheorem{thm}{Theorem}[section]
\newtheorem{lemma}[thm]{Lemma}

\newtheorem{cor}[thm]{Corollary}
\newtheorem{prop}[thm]{Proposition}

\newtheorem{rem}[thm]{Remark}

\numberwithin{equation}{section} 

\def\zz{\mathbb Z}

\def\rr{\mathbb R}

\def\fq{{\mathbb F}_q}
\def\ov{\overline}

\def\sm{\smallsetminus}

\def\la{\lambda}
\def\ga{\gamma}

\def\de{\delta}
\def\ep{\epsilon}
\def\al{\alpha}
\def\be{\beta}

\def\ve{\varepsilon}

\def\vp{\varphi}

\def\cG{\mathcal G}

\def\<{\langle}
\def\>{\rangle}

\def\SO{ {\text {\rm SO} } }

\def\St{ {\text {\rm St} } }

\def\GL{ {\text {\rm GL} } }
\def\SL{ {\text {\rm SL} } }
\def\Sp{ {\text {\rm Sp} } }

\def\Suz{ {\text {\rm Suz} } }
\def\PSU{ {\text {\rm PSU} } }

\def\U{{\text {\rm U} } }

\def\rK{\text{{\rm \textbf{K}}}}
\def\rLR{\text{{\rm \textbf{C}}}}
\def\rA{\text{{\rm \textbf{A}}}}
\def\sirr{\text{{\rm \textsc{W}}}}

\def\rD{\text{{\rm \textit{b}}}}

\def\rq{\text{{\rm \textit{f}}}}

\def\0{{\mathbf 0}}

\def\.{\hskip.06cm}
\def\ts{\hskip.03cm}

\def\vk{k}

\def\co{c_1}
\def\ct{c_2}

\begin{document}

\begin{abstract}
We give new bounds and asymptotic estimates on the
largest Kronecker and induced multiplicities of finite
groups.  The results apply to large simple groups of
Lie type and other groups with few conjugacy classes.
\end{abstract}


\maketitle

\section{Introduction} \label{sec:intro}

Given a finite group $G$, what is the largest dimension $\rD(G)$
of an irreducible complex representation of~$G$?   Which
representations attain it?  These questions are both fundamental
and surprisingly challenging.  For large simple groups of Lie type
they have been intensely studied especially in the last few years,
when asymptotic tools allowed for the general picture to emerge.
For $S_n$ and $A_n$, these questions are classical and have been the
subject of intense investigation for decades.  Despite some remarkable
successes the precise asymptotics is yet to be completely determined.
See Section~\ref{s:Sn} for precise statements and
$\S$\ref{ss:finrem-hist} for the references.

In recent years, Stanley initiated the study of the largest
\emph{Kronecker} and \emph{Littlewood--Richardson coefficients}
for the symmetric group
(see $\S$\ref{ss:finrem-Stanley}).  He computed their asymptotics
and asked to determine the characters which attain these asymptotics.
In our recent paper~\cite{PPY} we resolve both problems.  Perhaps
surprisingly, we show that the answer is always the asymptotically
largest degree, suggesting connection with the earlier work.

In this paper we generalize some of our results from $S_n$
to general finite groups with few conjugacy classes.  This is a
large class which includes quasisimple groups of Lie type of rank
$\ge 2$, large permutation groups, and even some nilpotent groups
of large class.

\smallskip

For a finite group $G$, the \emph{Kronecker multiplicity} \ts $g(\rho,\vp,\psi)$,
where $\rho,\vp,\psi \in \Irr(G)$, are defined by the equation:
\begin{equation}\label{eq:Kron-def}
\vp \cdot \psi \. = \. \sum_{\rho\in \Irr(G)} \. g(\rho,\vp,\psi)\. \rho\ts,
\end{equation}
where $\vp\cdot \psi$ is the usual product of characters:
\ts $[\vp \cdot \psi](x) = \vp(x)\ts \psi(x)$. Similarly,
for every subgroup $H<G$, $\rho \in \Irr(G)$ and $\pi \in \Irr(H)$,
we define the \ts \emph{induced multiplicities} \ts $c(\rho,\pi)$ by the equation:
\begin{equation}\label{eq:LR-def}
\Ind^G_H\. \pi \, = \, \sum_{\pi\in \Irr(H)} \. c(\rho,\pi) \. \pi\ts.
\end{equation}
While there is a great deal of literature for
determining these coefficients for classical
Chevalley groups like $\GL_n(q)$ and $\SO_{n}(q)$, very little
is known about their asymptotics.  Even less is known for other
types and other families of groups.  In this paper we obtain
bounds on the largest Kronecker and induced multiplicities
and illustrate them in many examples.

\smallskip

\subsection{Kronecker multiplicities}
Let $G$ be a finite group and let \ts $\vk(G)=|\Irr(G)|$ \ts denotes
the number of conjugacy classes of~$G$.
Define the \emph{largest Kronecker multiplicity} of~$G$:
$$\rK(G) \. := \. \max_{\rho,\vp,\psi \in \Irr(G)}  \. g(\rho,\vp,\psi)\ts.
$$

\begin{thm}\label{t:Kron-intro}
We have:
$$
\frac{\rD(G)^2}{\vk(G)^{1/2} \ts |G|^{1/2}} \, \le \, \rK(G) \, \le \, \rD(G)\ts.
$$
\end{thm}

Since $\ts \sqrt{|G|/\vk(G)}\le \rD(G)\le \sqrt{|G|}$, see~\eqref{eq:dim-groups},
this implies that for $\vk(G)$ small the
bound in the theorem is quite sharp.  The next result shows that $\rK(G)$ is
attained on characters of large degree in that case.

\begin{thm}  \label{t:Kron-main-intro}
Let \ts $\vp,\ts \psi\in \Irr(G)$. Suppose \ts
$\vp(1), \ts \psi(1) \ge \rD(G)/a$ \ts
for some \ts $a\ge 1$. Then there exists \ts
$\rho\in \Irr(G)$, such that:\
$$
\rho(1) \, \ge \, \frac{\rD(G)}{a\cdot \vk(G)^{1/2}}\.
\quad \text{and} \quad
g(\rho,\vp,\psi) \, \ge \, \frac{\rD(G)}{a^2\cdot \vk(G)}\..
$$
\end{thm}

\smallskip

\subsection{Induced multiplicities}
Let $H<G$.  For all $\rho \in \Irr(G)$ and
$\pi \in \Irr(H)$, define the \emph{largest induced multiplicity}:
$$
\rLR(G,H) \. := \. \max_{\rho \in \Irr(G)} \. \max_{\pi \in \Irr(H)} \. c(\rho,\pi).
$$

\begin{thm} \label{t:LR-intro-bounds}
Let $H<G$.  Then:
$$
\frac{1}{\vk(H)^{1/2} \ts \vk(G)^{1/2}} \. \ts [G:H]^{1/2} \, \le \,
\rLR(G,H) \, \le \, [G:H]^{1/2} \ts.
$$
\end{thm}

In other words, when $\vk(H),\vk(G)$ are small, the largest induced multiplicities
are close to~$\sqrt{[G:H]}$.  The following result again shows that large
induced multiplicities are attained at characters of large degree.

\begin{thm} \label{t:LR-intro-exist}
Let \ts $H<G$ and \ts $\rho\in \Irr(G)$.  Suppose \ts
$\rho(1)\ts \ge \ts |G|^{1/2}/a$, for some \ts $a\ge 1$.  Then there exists
\ts $\pi\in\Irr(H)$, such that:
$$
\pi(1) \, \ge \, \frac{|H|^{1/2}}{a \cdot \vk(H)}
\quad \text{and} \quad
c(\rho,\pi) \, \ge \, \frac{[G:H]^{1/2}}{a \cdot \vk(H)}\,.
$$
\end{thm}

\begin{rem}\label{r:intro-diag} {\rm
Note that Kronecker multiplicities are a special case of
induced multiplicities.  To see this, take $G=H\times H$ and
a diagonal subgroup $H<G$, and we have \ts $\rLR(H \times H, H) = \rK(H)$.
Observe that the bounds for $\rK(H)$ which follow from Theorem~\ref{t:LR-intro-bounds}
in this case are weaker than the bounds in Theorem~\ref{t:Kron-intro}
(see Remark~\ref{r:Kron-comparison}). This follows from the dependence
of $\rLR(G,H)$ on the embedding $H \hookrightarrow G$.  For example,
for $G=H\times H$ as above and $(H \times 1)\hookrightarrow G$,
we have \ts $\rLR(H \times H, H) = \rD(H)$, which can be much larger
than $\rK(H)$ (see~$\S$\ref{s:ex}).
}\end{rem}

\smallskip
\begin{rem}{\rm \label{r:intro-Sn}
As we mentioned earlier, for the symmetric groups
$G=S_n$ and $H=S_k\times S_{n-k}$ the Kronecker and induced multiplicities
are called the Kronecker and the Littlewood--Richardson coefficients,
respectively. They play a crucial role in Algebraic Combinatorics
and its applications, and have been intensely studied from both enumerative,
algebraic, geometric, probabilistic and computational point of view
(see $\S$\ref{ss:finrem-Stanley} and~\cite{PPY} for the references).
}
\end{rem}

\medskip

\subsection*{Structure of the paper}
In sections~\ref{s:conj} and~\ref{s:degree}, we review
known bounds on \ts $\vk(G)$ and $\rD(G)$, respectively,
for various classes of groups. In Section~\ref{s:Sn} we
discuss the symmetric group case and our state of knowledge
on $\rD(S_n)$.  Then, in sections~\ref{s:sym} and~\ref{s:ex},
we apply our bounds to various examples of groups and subgroups.
We prove theorems~\ref{t:Kron-intro} and~\ref{t:Kron-main-intro}
in Section~\ref{s:Kron}.  We then prove our theorems~\ref{t:LR-intro-bounds}
and~\ref{t:LR-intro-exist} in Section~\ref{s:LR}.  We conclude with
open problems and final remarks (Section~\ref{s:fin-rem}).

\medskip

\subsection*{Notation}
Most our notation are standard.  Let \ts $\Irr(G)$ \ts  denotes the set of
\emph{irreducible characters} of~$G$, let \ts $\Conj(G)$ \ts be the set of conjugacy
classes, and \ts $\vk(G)=|\Irr(G)|=|\Conj(G)|$ the \emph{number of conjugacy classes}.
By $|C_G(x)|$ we denote  the size of the centralizer of element $x\in G$.

\bigskip

\section{Number of conjugacy classes}\label{s:conj}

\subsection{General bounds}\label{ss:conj-gen}
There are many general lower and upper bounds for $\vk(H)$; we will only
mention some key results but will not be able to review it.  The
subject was initiated by E.~Landau in~1903 with the first quantitative bound
\ts $\vk(H) = \Omega\bigl(\log \log |H|\bigr)$ \ts by Erd\H{o}s and Tur\'an (1968).  Recently,
Jaikin-Zapilrain~\cite{JZ} showed the first super-log lower bound for nilpotent groups,
but for general finite groups there is only a sub-log bound due to Pyber~\cite{Pyb},
slightly improved in~\cite{BMT,Kel}.  For a nilpotent group $H$ of bounded
class~$r$, Sherman~\cite{She} proved:
\begin{equation} \label{eq:sherman}
\vk(H) \. \ge \. r\ts |H|^{1/r} \ts -\ts r \ts + \ts 1\ts.
\end{equation}

In a different direction,
for a permutation group $H < S_n$, Kov\'acs and Robinson~\cite{KR} showed that \ts
$\vk(H) \le 5^n$.  This was improved to \ts $2^{n-1}$ \ts in~\cite{LP}, and further to
\ts $\vk(H) \le 5^{(n-1)/3}$ for $n\ge 4$, \ts in~\cite{GM}.

Finally, there are general upper and lower bounds on the number of conjugacy classes,
notably:
\begin{equation} \label{eq:Gal}
\frac{\vk(H)}{[G:H]} \, \le \, \vk(G) \, \le \,\vk(H) \cdot [G:H] \quad \text{for} \ \, H<G,
\end{equation}
see~\cite{Gal}. Sometimes these bounds are written in terms of the \emph{commuting probability}.
Notably, Guralnick and Robinson~\cite{GR} prove
$$\vk(G)\. \le \. \sqrt{|G| \. \vk(F)}\.,
$$
where $F$ is the Fitting subgroup of~$G$.  In particular, $\ts \vk(G) \le \, \sqrt{|G|}$
when $Z(G)=1$, i.e.\ when the center $G$ is trivial.

\medskip

\subsection{Number of conjugacy classes for groups of Lie type}
It was shown by Liebeck and Pyber~\cite{LP} that
for a completely reducible subgroup $G < \GL_n(q)$,
we have \ts $\vk(G) \le q^{10 \ts n}$.
Further, when $G$ is a quasisimple group of Lie type
over $\fq$ of rank~$r$, they show \ts $\vk(G) \le (6\ts q)^{r}$.
Fulman and Guralnick~\cite{FG} further improve these bounds
to
\begin{equation}\label{eq:conj-FG}
q^r \. \le \. \vk(G) \. \le \. 27.2\ts q^{r}\ts,
\end{equation}
with better constants in special cases.  In fact,
in many cases either sharp asymptotic
bounds, or even the exact formulas are known, see examples in
Section~\ref{s:ex}.


\bigskip

\section{Largest degree}\label{s:degree}

\subsection{General bounds}\label{ss:degree-gen}
Let $\rho(1)$ denote the \emph{degree} of~$\rho$, and let
$$\rD(G) \. :=  \. \max_{\rho \in \Irr(G)} \ts \rho(1)
$$
denote the largest degree.  Recall the \emph{Burnside identity}:
\begin{equation}\label{eq:Burnside}
\sum_{\rho\in \Irr(G)} \. \rho(1)^2 \. = \. |G|\ts.
\end{equation}
This immediately implies that for all finite groups~$G$, we have:
\begin{equation}\label{eq:dim-groups}
\sqrt{|G|/\vk(G)} \. \le \. \rD(G)\le \sqrt{|G|}\..
\end{equation}
There are a few lower and upper bounds for general groups.  Notably,
if \ts $\rD(G) < \sqrt{|G|}$, then
$$
\rD(G) \, \le \, \sqrt{|G|}\. - \.\frac12\.\sqrt[4]{|G|}\ts,
$$
and this is the best bound of this type~\cite{HLS},
improving on earlier bounds by Isaacs~\cite{Isa}
and others.

One should, of course, expect better upper bounds for large non-solvable groups.
For example, if $\rho\in \Irr(G)$ and $\rho(1)^2\ge |G|/2$, then $\rho^2$ contains
every irreducible character, i.e.\ $g(\rho,\rho,\chi)>0$ \ts for all \ts
$\chi\in  \Irr(G)$.  This property is known for all simple groups of
Lie type~\cite{HSTZ} except for $\PSU_n(q)$, and is a subject of intense
study for~$A_n$ and~$S_n$, see~\cite{Ike,LuS,PPV}. In the opposite direction,
for all simple groups one has \ts $\rD(G)\ge \sqrt[3]{|G|}$, see~\cite{KS}.

\subsection{Largest degree for groups of Lie type}\label{ss:degree-Lie}
For a natural class of reductive linear algebraic groups $G$ of dimension~$d$, rank~$r$
over~$\fq$, Kowalski~\cite[Prop.~5.5]{Kow} uses an argument by J.~Michel
to prove:
$$
\rD(G) \, \le \, \frac{|G|}{(q-1)^r \ts |G|_p} \, \le \, (q+1)^{(d-r)/2},
$$
where $N_p$ denotes the largest power of~$p$ which divides~$N$.  He also proves
that the first inequality is sharp when $q=p^a$ is large enough, and
obtains explicit bounds for several series, such as $\GL_n(q)$,
 $\Sp_{2n}(q)$, etc.

More general and sometimes more precise bounds were obtained later
by Larsen, Malle and Tiep~\cite{LMT}. For all $G(q)$ over~$\fq$, of dimension~$d$, rank~$r$,
characteristic~$p$, they prove:
\begin{equation}\label{eq:dim-LMT}
A\ts  (\log_q r)^\al \. |G|_p\, \le \, \rD(G) \, \le \, B\ts  (\log_q r)^\be \. |G|_p \,,
\end{equation}
for some universal constants \ts $A,B>0$ and $\al,\be\ge 0$.  In fact, the $\log$
terms disappear for exceptional groups of Lie type.  They also obtain sharp explicit
bounds in special cases, see Section~\ref{s:ex}.  In full generality, we obtain the
following result.

\begin{thm} \label{t:Kron-Lie}
Let $\cG$ be a simple algebraic group of characteristic~$p$, rank~$r$, and
a finite group $G(q):=\cG^F$ over $\fq$, corresponding to a Frobenius map $F: \cG\to \cG$.
Then:
$$ C \. \frac{\bigl(|G|_p\bigr)^2 \ts (\log_q r)^{\ga}}{q^{r/2} \. \sqrt{|G|} }
\, \le \, \rK\bigl(G(q)\bigr) \, \le \, D\ts  (\log_q r)^\de \. |G|_p
$$
where $C,D>0$ and $\ga,\de\ge 0$ are universal constants independent of~$\cG$ and~$q$.
\end{thm}

\begin{proof} In Theorem~\ref{t:Kron-intro}, use bounds on $\vk(G_n)$ and
$\rD(G_n)$ in~\eqref{eq:conj-FG} and~\eqref{eq:dim-LMT}, respectively.
\end{proof}

\bigskip

\section{Symmetric groups}\label{s:Sn}

\subsection{Largest degree}\label{s:Sn-ve}
Recall that $\vk(S_n) = p(n)$, the number of integer partitions
of~$n$.
The \emph{Hardy--Ramanujan asymptotic formula} gives:
$$p(n) \. \sim  \. \frac{1}{4\ts n \ts
\sqrt{3}} \, e^{\pi\ts \sqrt{\frac{2n}{3}}} \quad \text{as} \ \ n\to \infty.
$$
In 1985, Vershik and Kerov~\cite{VK2} proved that for all $n$ large enough:
\begin{equation}\label{eq:VK}
\sqrt{n!} \. e^{-\co\ts \sqrt{n}\ts (1+o(1))} \, \le \, \rD(S_n) \, \le \,  \sqrt{n!} \. e^{-\ct\ts \sqrt{n}\ts (1+o(1))} ,
\end{equation}
where
\begin{equation}\label{eq:VK-const}
\co \. = \. \pi\ts \sqrt{\frac{1}{6}} \. \approx \ts 1.2825 \quad
\text{and} \quad \ct \. = \. \frac{\pi-2}{\pi^2} \. \approx \ts 0.1157
\end{equation}
Note that the lower bound follows from~\eqref{eq:dim-groups},
but the upper bound is rather remarkable.  The following result
is an application.

\medskip

\subsection{Smaller degrees}\label{s:Sn-ve}
Let \ts $\sirr(G) = \{\rho\in \Irr(G), \, \rho(1)<\rD(G)\}$, and let
$\ve(G)$ be defined as follows:
$$
\ve(G) \, = \, \frac{\sum_{\rho\in \sirr(G)} \. \rho(1)^2}{\rD(G)^2}\ts.
$$
One can think of $\ve(S_n)$ as the ratio of probability of non-largest to largest
characters of $S_n$ w.r.t.\ the Plancherel measure, see~\cite{Bia,Rom,VK1}.
In~\cite{LMT}, the authors show
that $\ve(S_n) = \Omega(1)$.  In fact, they prove that there exist a universal
constant \ts $\ep>0$ \ts s.t.\ $\ve(G)>\ep$ for all non-abelian finite simple groups~$G$.
The former result was improved in~\cite{HHN} to $\ve(S_n) =\Omega(n)$.

\begin{thm} \label{t:Sn-ve}
There exist universal constants \ts $a_2> a_1 > 0$, such that:
$$e^{a_1\sqrt{n}} \, \le \,
\ve(S_n) \, \le \, e^{a_2\sqrt{n}}
$$
\end{thm}

\begin{proof}  For the upper bound, we have:
$$
\ve(S_n) \, \le \, \frac{n!}{\rD(S_n)}\ts,
$$
and the result follows from the lower bound in~\eqref{eq:VK}.
For the lower bound, let $M(n)$ denote the number of characters
of the largest degree, i.e.\
$M(n) = |\Irr(S_n)\sm\sirr(S_n)|$.  It was proved in
\cite[Prop.~3.5(1)]{HHN}
that \ts $\ve(S_n)\ge M(n)/16$.  On the other hand, from the upper bound
in~\eqref{eq:VK}, we have:
$$
\ve(S_n) \, = \, \frac{n! \, - M(n)\ts \rD(S_n)^2}{\rD(S_n)^2}
\, \ge \,  e^{2\ct\ts \sqrt{n}\ts (1+o(1))} \. - \. M(n),
$$
and the result follows by combining these two inequalities.
\end{proof}

\begin{rem} {\rm
We conjecture that the sequence \ts $M(1),M(2),\ldots$ \ts is bounded.
In fact, this is an interesting question for non-largest
degrees as well.  Let
$$
\rq(n) \, := \, \max_{k} \,\. \bigl|\bigl\{\la\vdash n~:~\chi^\la(1)=k\bigr\}\bigr|\ts.
$$
For example, $\rq(13)=6$ \ts since \ts $d(94)=d(76)=d(10\ts\ts 21)=d(321^8)=d(2^61)=d(2^41^5)=429$,
where $d(\la):=\chi^\la(1)$.  We conjecture that the sequence \ts $\rq(1),\rq(2),\ldots$ \ts
is unbounded.
}
\end{rem}
\bigskip

\section{General linear groups}\label{s:sym}

\subsection{Bounds on Kronecker multiplicities}
For \ts$G_n =\GL_n(q)$, there are sharp bounds on all parameters we need.
We have:
$$\left(1\. - \. \frac1q - \frac{1}{q^2}\right)\ts q^{n^2} \. \le \. |G_n| \. \le \. q^{n^2}\ts,
$$
where the first inequality is given in~\cite{Pak}. Similarly,
$$q^{n}\. - \. q^{n-1} \. \le \. \vk(G_n) \. \le \. q^{n}\ts,
$$
where the lower bound follows from~\eqref{eq:conj-FG} and upper bound is given
in~\cite[Lemma~5.9(ii)]{MR}.
Finally,
$$\frac14 \ts \bigl(1+\log_q (n+7)/2\bigr)^{3/4} \ts q^{n(n-1)/2}\, \le \, \rD(G_n) \, \le \,  13 \ts \bigl(1+\log_q (n+1)\bigr)^{2.54}\ts q^{n(n-1)/2}\ts,
$$
see~\cite[Thm.~5.1]{LMT}.\footnote{There does
not seem to be a closed formula for $\rD(\GL_n(q))$, however the \emph{Steinberg character}
\ts $\St$ \ts is asymptotically the largest unipotent irreducible character;
see discussion in~\cite[$\S$5]{LMT}
(cf.~\cite{HSTZ}).}  Theorem~\ref{t:Kron-intro} then gives the upper and lower bounds:
$$
\frac1{16} \ts \bigl(1+\log_q (n+7)/2\bigr)^{3/2} \ts q^{n(n-3)/2}\, \le \,\rK(G_n)
\, \le \, 13 \ts \bigl(1+\log_q (n+1)\bigr)^{2.54}\ts q^{n(n-1)/2}\ts.
$$

\medskip

\subsection{Induced multiplicities from a block subgroup}
Let $q$ be fixed, $n=2m$, $n\to \infty$, and let $G_n =\GL_n(q)$ be as above.
Consider a subgroup $H_n:=(G_m\times G_m)$ of~$G_n$  of index
\ts $[G_n:H_n] = q^{\frac{n^2}{2} \ts +\ts O(n)}$.
Clearly, $\vk(H_n) = \vk(G_m)^2$.  Theorem~\ref{t:LR-intro-bounds} then gives:
$$
\LR(G_n, H_n) \, = \, q^{\frac{n^2}{4} \ts +\ts O(n)}\ts.
$$

\medskip

\subsection{Induced multiplicities from a parabolic subgroup}
Similarly, let $n=2m$, $G_n =\GL_n(q)$, and let \ts $B_n<G_n$ \ts
be a subgroup of matrices \ts $(x_{ij})\in G_n$ \ts
with $x_{ij} =0$ for all $i>m$, $j\le m$. Thus $[G_n:B_n] = q^{\frac{n^2}{4} \ts +\ts O(n)}$.
We also have \ts $\vk(B_n) = q^{O(n)}$.  To prove this, take a normal subgroup $A_n$ of
upper right $m\times m$ matrices, $B_n/A_n\simeq H_n$ as above, and consider the action
of $B_n$ on $A_n$.  Then use the exact formula in~\cite{FF} and estimates in~\cite{Pak}
(we omit the details).  Now  Theorem~\ref{t:LR-intro-bounds} gives:
$$
\LR(G_n, B_n) \, = \, q^{\frac{n^2}{8} \ts +\ts O(n)}\ts.
$$

\bigskip

\section{Further examples}\label{s:ex}

\subsection{Linear groups of rank~1}\label{ss:SL2p} Let \ts
$G_p := \SL_2(p)$, where $p$ is a prime.
Then:
$$|G_p| \ts =\ts p^3-p\ts, \quad \vk(G_p)\ts =\ts p+4\ts,
\quad \rD(G_p) = p+1\ts.
$$
In this case the whole character table can be computed by hand, so the lower
bound in Theorem~\ref{t:Kron-intro} is neither sharp nor useful.

\medskip

\subsection{Suzuki groups}\label{ss:Suz} Let \ts $G_n :=  \Suz(q)$,
where \ts $q=2^{2n+1}$ and \ts $n\to \infty$.  Then:
$$|G_n| \ts =\ts q^2(q^2+1)(q-1)\ts,
\ \, \vk(G) \ts =\ts q+3\ts, \ \ \, \rD(G_n) = q^2 + O(q^{3/2}).
$$
By Theorem~\ref{t:Kron-intro} we have a sharp bound:
\ts $\rK(G_n) = q^2 + O(q^{3/2})$.

\medskip

\subsection{Unitriangular groups}  Let \ts $H_n = \U_n(q)$ \ts be the group of
upper triangular matrices with ones on the diagonal.  Let $q$ be fixed and
$n\to \infty$.  We have:
$$
|H_n| \ts =\ts q^{\binom{n}{2}}, \quad q^{\frac{n^2}{12} \ts +\ts O(n)} \. \le  \. \vk(H_n) \. \le \. q^{\frac{7n^2}{44} \ts +\ts O(n)}, \quad
\rD(H_n) \, =\. q^{\mu(n)},
$$
where $\mu(n) = \lfloor (n-1)^2/4\rfloor$, see~\cite{Isa2}.  Here the
lower bound on $\vk(H_n)$ is by Higman~\cite{Hig}, and the upper bound
on $\vk(H_n)$ is by Soffer~\cite{Sof}.  Note that Sherman's
bound~\eqref{eq:sherman} is quite weak in this case.  Similarly,
the lower bound in~\eqref{eq:dim-groups} is very weak in this case,
while the upper bound is quite sharp.  Theorem~\ref{t:Kron-intro} then gives:
$$
q^{\frac{n^2}{11} \ts +\ts O(n)}\. \le  \. \rK(H_n) \. \le  \. q^{\frac{7\ts n^2}{44} \ts +\ts O(n)}\ts.
$$
It would be interesting to see if this bound can be improved, perhaps, by using
the supercharacter theory, see~\cite{DI,Yan}.

\medskip

\subsection{Unitriangular subgroup}
Let $q$ be fixed.  In notation above, note that
\ts $H_n= \U_n(q)$ is a subgroup of $G_n =\GL_n(q)$ \ts of index
\ts $[G_n:H_n] = q^{\frac{n^2}{4} \ts +\ts O(n)}$, \ts as \ts $n\to \infty$.
Theorem~\ref{t:LR-intro-bounds} then gives:
$$
q^{\frac{15\ts n^2}{88} \ts +\ts O(n)}\, \le \, \LR(G_n, H_n) \, \le \, q^{\frac{n^2}{4} \ts +\ts O(n)}\ts.
$$

\medskip

\subsection{The Monster group}\label{ss:monster}
Let $M$ be the \emph{Monster group}.  We have:
$$
|M| \approx 8.08 \cdot 10^{53}, \qquad \vk(M)=194,
\qquad \rD(M) \approx 2.59 \cdot 10^{26}\ts.
$$
In notation of~$\S$\ref{s:Sn-ve}, we have \ts $\ve(M) \approx 11.02$, which
follows from many ``large but not largest'' irreducible characters.
On the other hand, equation~\eqref{eq:dim-groups} gives:
$$\sqrt{|M|/\vk(M)} \approx 6.45\cdot 10^{25} \, \le \,
\rD(M) \approx 2.59 \cdot 10^{26} \, \le \, \sqrt{|M|} \approx 8.99 \cdot 10^{26}\ts.
$$
The large gap in the first inequality can be explained by a large number
of characters with very small degree.

We compare the exact value of \ts $\rK(M)$ \ts computed directly from
the character table~\cite{atlas}, with estimates in Theorem~\ref{t:Kron-intro}:
$$\frac{\rD(M)^2}{\sqrt{\vk(M)\ts |M|}} \approx 5.35 \cdot 10^{24}
\, \le \, \rK(M)  \approx 2.15 \cdot 10^{25} \, \le \, \rD(M) \approx 2.59 \cdot 10^{26}\ts.
$$
Again both gaps can be similarly explained by the presence of many
``relatively small'' irreducible characters which allow the
isotypical components to be relatively evenly distributed (cf.\
the proof of Theorem~\ref{t:Kron-intro} in $\S$\ref{ss:Kron-max}).
In fact, $\rK(M)$ is much larger than the \emph{average Kronecker multiplicity}:
$$
\frac{1}{\vk(M)^3} \. \sum_{\rho,\vp,\psi\in \Irr(M)} \ts g(\rho,\vp,\psi) \ts \approx \ts 3.38\cdot 10^{22}\ts,
$$
which can be explained by the fact that if even one of the three characters
has small degree, then so does \ts $g(\rho,\vp,\psi)$, see~\eqref{eq:Kron-upper-groups}.

In fact, the bound $\rA(G)\ge |G|$ (see next section) is unusually tight
in this case:\footnote{See A.~Hulpke's answer in \ts
\url{https://math.stackexchange.com/questions/2668042}}
$$
\aligned
\rA(M) \ \ \  = \  808017424794512875894769468067441075690144312450960558 & \\
|M| \quad  = \     808017424794512875886459904961710757005754368000000000 &
\endaligned
$$
There is a simple explanation, of course: the centralizer sizes $z_\al$ rapidly
decrease as we go down the list.  Here are the first three of them other than
$z_1=|M|$, corresponding to the three largest maximal subgroups of~$M$:
$$
2\ts |B| \approx    8.31\cdot 10^{33}, \qquad
2^{25} \ts |Co_1|  \approx    1.40 \cdot 10^{26},
\qquad 3\ts |Fi_{24}|  \approx    3.77 \cdot 10^{24}.
$$
When these are subtracted from $\rA(M)$ we obtain a relatively small remainder:
$$
\rA(M)\ts -\ts |M| \ts - \ts 2\ts |B| \ts - \ts 2^{25} \ts |Co_1|
\ts - \ts 3\ts |Fi_{24}| \. \approx \. 1.00 \cdot 10^{19}\ts.
$$

\bigskip

\section{Kronecker multiplicities}\label{s:Kron}

\subsection{General inequalities}\label{ss:Kron-gen}
First, note:
$$
g(\rho,\vp,\psi) \. = \. \bigl\<\rho,\ts \vp \cdot \psi\bigr\> \. = \. \bigl\<\ov{\rho}\cdot \vp \cdot \psi, 1\bigr\>.
$$
This implies the symmetries
\begin{equation}\label{eq:Kron-sym-groups}
g(\rho,\vp,\psi) \. = \. g(\ov{\vp},\ov{\rho},\psi) \. = \.g(\ov{\vp},\psi,\ov{\rho})  \. = \. \ldots
\end{equation}
In particular, we have a general upper bound:
\begin{equation}\label{eq:Kron-upper-groups}
g(\rho,\vp,\psi) \. \le \. \rho(1) \cdot \min\bigl\{\vp(1)/\psi(1), \ts \psi(1)/\vp(1)\bigr\} \. \le \. \rho(1)\ts.
\end{equation}

\smallskip

\begin{prop}  \label{p:Kron-main-groups}
Let \ts $\rho,\vp,\psi\in \Irr(G)$. Suppose \ts $g(\rho,\vp,\psi) \ge \rD(G)/a$, for some $a\ge 1$.
Then: \ts $\rho(1),  \vp(1), \psi(1) \ge \rD(G)/a$.
\end{prop}

\begin{proof}  This follows immediately from~\eqref{eq:Kron-upper-groups}
and the symmetries~\eqref{eq:Kron-sym-groups}.
\end{proof}

\medskip

\subsection{Largest Kronecker multiplicity} \label{ss:Kron-max}
Recall the definition of $\rK(G)$ given in the introduction.  Let
\begin{equation}\label{eq:sum-squares-def}
\rA(G) \, := \,
\sum_{\rho,\vp,\psi \in \Irr(G)} \. g(\rho,\vp,\psi)^2.
\end{equation}

\begin{lemma}\label{l:kron-sum-squares}
We have:
\begin{equation}\label{eq:Kron-squares-groups}
\rA(G) \, = \, \sum_{\al \in \Conj(G)} \. z_\al\ts,
\end{equation}
where $z_\al = |C(\al)|$ is the size of the centralizer of an element $x\in \al$.
\end{lemma}

\begin{proof}[Proof of Lemma~\ref{l:kron-sum-squares}]
By definition, we have:
$$g(\rho,\vp,\psi) \, = \, \frac{1}{|G|} \. \sum_{x \in G} \. \ov{\rho(w)}\ts \vp(x)\ts\psi(x)\ts = \, \frac{1}{|G|} \. \sum_{x \in G} \. \rho(w)\ts \ov{\vp(x)}\ts\ov{\psi(x)}\ts,
$$
noting that $\chi(g^{-1})=\ov{\chi(g)}$ for finite groups.
Hence, we can write the sum of squares as
$$\begin{aligned}
\rA(G) \, & = \, \sum_{\rho,\vp,\psi \in \Irr(G)} \. g(\rho,\vp,\psi)^2 \, = \, \frac{1}{|G|^2} \. \sum_{x,y \in G} \sum_\rho \. \ov{\rho(x)}\ts\rho(y) \. \sum_\vp \. \vp(x)\ts\ov{\vp(y)} \. \sum_{\psi} \. \psi(x)\ts \ov{\psi(y)}\\
&  = \, \frac{1}{|G|^2} \. \sum_{x,y\in G} \left(\sum_\rho \rho(x)\ts \ov{\rho(y)}\right)^3 \, = \, \frac{1}{|G|^2} \. \sum_{\al \in \Conj(G)} \. \left(\frac{|G|}{z_\al}\right)^2 \ts (z_\al)^3\, = \, \sum_{\al \in \Conj(G)} \. z_\al\..
\end{aligned}
$$
Here the last equality follows from orthogonality of the columns in the character table.
\end{proof}

\begin{prop}  \label{p:Kon-max-groups}
We have:
$$
\frac{|G|^{1/2}}{\vk(G)^{3/2}} \, \le \, \rK(G) \, \le \, \rD(G)\ts.
$$
\end{prop}

\begin{proof}
In~\eqref{eq:Kron-squares-groups}, we have \ts $\rA(G) \ge z_1=|G|$.
This gives the lower bound.  
The upper bound follows from~\eqref{eq:Kron-upper-groups}.
\end{proof}

Theorem~\ref{t:Kron-main-intro} can be viewed as a converse of
Proposition~\ref{p:Kron-main-groups}.

\begin{proof}[Proof of Theorem~\ref{t:Kron-main-intro}]
Let $\rho$ be the character in the largest term in the RHS of
$$
\frac{\rD(G)^2}{a^2} \, \le \, \vp(1) \cdot \psi(1) \. = \. \sum_{\rho\in \Irr(G)} \. g(\rho,\vp,\psi)\. \rho(1)\ts.
$$
On the one hand,
$$
g(\rho,\vp,\psi) \, \ge \, \frac{1}{\vk(G)\cdot \rD(G)} \ts\cdot \ts \frac{\rD(G)^2}{a^2}\, = \, \frac{\rD(G)}{a^2\ts \vk(G)}\..
$$
On the other hand,
$$
\rho(1)^2 \, \ge \, g(\rho,\vp,\psi)\. \rho(1) \, \ge \, \frac{1}{\vk(G)} \ts\cdot \ts\frac{\rD(G)^2}{a^{2}}\.,
$$
which implies the result.
\end{proof}

\medskip

\subsection{Refined Kronecker multiplicities} \label{ss:Kron-ref}
Fix $\rho,\vp\in \Irr(G)$.  Define the \emph{largest refined Kronecker multiplicity}
$$
\rK(G;\ts\rho,\vp) \. := \. \max_{\psi \in \Irr(G)}  \. g(\rho,\vp,\psi)\.
$$
Clearly, $\rK(G;\ts\rho,\vp) \le \rK(G)$.

\begin{prop}\label{p:Kron-refined-groups}  For all $\rho,\vp\in \Irr(G)$, we have:
$$
\frac{\rho(1) \. \vp(1)}{\vk(G)^{1/2} \ts |G|^{1/2}} \, \le \,\rK(G; \ts \rho,\vp) \, \le \, \min\bigl\{\rho(1),\ts \vp(1)\bigr\}.
$$
\end{prop}

\begin{proof} Let
$$
A(\rho,\vp) \, := \, \sum_\psi \. g(\rho,\vp,\psi)^2\ts.
$$
Recall Burnside's identity~\eqref{eq:Burnside} and
$$
\sum_\psi \. g(\rho,\vp,\psi) \ts \psi(1) \, = \, \vp(1)\ts \rho(1)\ts.
$$
Now apply the Cauchy--Schwarz inequality to vectors \.
$\bigl(\psi(1)\bigr), \bigl(g(\rho,\vp,\psi)\bigr)\in \rr^{\vk(G)}$,
both indexed by $\psi \in \Irr(G)$.  We obtain:
$$
A(\rho,\vp) \, \ge \, \frac{\rho(1)^2\ts \vp(1)^2}{|G|}\..
$$
Therefore, for the maximal term in the summation $A(\rho,\vp)$, we have:
$$
\max_\psi \ts g(\rho,\vp,\psi) \, \ge \,
\frac{\rho(1)\ts \vp(1)}{\vk(G)^{1/2} \ts |G|^{1/2}}\,.
$$
This implies the lower bound.  The upper bound follows from~\eqref{eq:Kron-upper-groups}.
\end{proof}

\begin{proof}[Proof of Theorem~\ref{t:LR-intro-bounds}]
  In Proposition~\ref{p:Kron-refined-groups}, take
\ts $\rho,\vp \in \Irr(G)$ \ts s.t.\ \ts $\rho(1)=\vp(1) = \rD(G)$.
\end{proof}

\begin{rem}\label{r:Kron-comparison}
{\rm Note that Theorem~\ref{t:LR-intro-bounds}
and equation~\eqref{eq:dim-groups} imply the lower bound in
Proposition~\ref{p:Kon-max-groups}.  In fact, the latter
lower bound is same bound mentioned in Remark~\ref{r:intro-diag}
for the diagonal subgroup $\rK(H) = \rLR(H\times H,H)$.
}\end{rem}

\bigskip

\section{Induced multiplicities} \label{s:LR}

\subsection{General inequalities} \label{ss:LR-gen}
Let $H<G$ be a subgroup of a finite group $G$ of
index $[G:H] = |G|/|H|$.  For all $\rho \in \Irr(G)$ and
$\pi \in \Irr(H)$, define
the \ts \emph{induced multiplicities} \ts $c(\rho,\pi)$ \ts as follows:
$$c(\rho,\pi) \, := \, \bigl\<\rho,\ts \pi{\upa}^G_H\bigr\>  \, = \,
\bigl\<{\rho\text{\hskip-0.04cm}}\doa^G_H, \ts \pi\bigr\>\ts.
$$
We have:
\begin{equation}\label{eq:LR-groups-def}
\sum_{\rho \in \Irr(G)} \.  c(\rho,\pi) \. \rho(1) \. = \. [G:H]\cdot \pi(1)\quad \ \ \text{and}  \ \
\quad \sum_{\pi \in \Irr(H)} \.  c(\rho,\pi) \. \pi(1)\. = \.\rho(1)\ts.
\end{equation}

\begin{lemma}  \label{l:LR-squared-groups}
For every $H < G$, we have:
\begin{equation}\label{eq:LR-squares}
\sum_{\rho \in \Irr(G)} \. \sum_{\pi \in \Irr(H)} \. c(\rho,\pi)^2 \, = \, \sum_{\al \in \Conj(H)} \. \frac{z_\al(G)}{z_{\al}(H)}\ts,
\end{equation}
where \ts $z_\al(H)=|C_H(x)|$ denotes the size of the centralizer of $x \in \al$
within~$H$, and $z_\al(G)=|C_G(x)|$ is the size of the centralizer within~$G$.
\end{lemma}

\begin{proof}  Denote by \ts $\xi = \rho|_H$ \ts  the restriction of the character $\rho$ to~$H$.
We have:
$$c(\rho,\pi) \, = \, \sum_{\al \in \Conj(H)}\. z_\al^{-1} \ts \xi(\al) \ts \ov{\pi(\al)}\ts,
$$
Then:
$$
\aligned
\sum_{\rho \in \Irr(G)} \. \sum_{\pi \in \Irr(H)} \. c(\rho,\pi)^2  \, & = \,
\sum_{\al,\gamma \in \Conj(H)} \. z_\al^{-1}\ts z_\gamma^{-1} \. \sum_{\rho \in \Irr(G)} \. \sum_{\pi \in \Irr(H)} \.  \xi(\al)\ts \ov{\xi(\gamma)} \.  \ov{\pi(\al)}\ts \pi(\gamma)\\
& = \, \sum_{\al \in \Conj(H)} z_\al(H)^{-2} \ts \bigl(z_\al(H)\cdot z_\al(G)\bigr) \,  = \, \sum_{\al \in \Conj(H)} \. \frac{z_\al(G)}{z_{\al}(H)}\.,
\endaligned
$$
as desired.
\end{proof}

\begin{cor}  \label{c:LR-squared-groups-ineq}
For every $H < G$, we have:
$$
\sum_{\rho \in \Irr(G)} \. \sum_{\pi \in \Irr(H)} \. c(\rho,\pi)^2 \, \ge \,[G:H] \ts.
$$
\end{cor}
\begin{proof}
Since in \eqref{eq:LR-squares} the \ts RHS~$\ge z_1(G)/z_1(H) = [G:H]$, we obtain the inequality.
\end{proof}

\begin{rem}\label{r:LR-identity}
{\rm  Note that Lemma~\ref{l:kron-sum-squares} easily follows from
Lemma~\ref{l:LR-squared-groups} by taking the diagonal subgroup in
$G\times G$ as in Remark~\ref{r:intro-diag}.  The details are
straightforward.  We chose to keep both proofs for clarity of
exposition.
}\end{rem}

\begin{lemma}  \label{l:LR-squared-groups-2}
For every $H < G$, we have:
$$\sum_{\rho \in \Irr(G)} \.  c(\rho,\pi)^2 \, \le \,  [G:H] \quad \ \text{and} \ \quad
\sum_{\pi \in \Irr(H)} \. c(\rho,\pi)^2 \, \le \,  [G:H]\ts.$$
\end{lemma}

\begin{proof}  We have:
$$\sum_{\rho \in \Irr(G)} \.  c(\rho,\pi)^2 \, \le \, \sum_{\rho \in \Irr(G)} \.
c(\rho,\pi) \. \frac{\rho(1)}{\pi(1)} \, = \,  \frac{1}{\pi(1)} \ts \cdot \ts \pi(1)\ts [G:H]
\, = \,  [G:H]\ts,
$$
$$\sum_{\pi \in \Irr(H)} \.  c(\rho,\pi)^2 \, \le \, \sum_{\pi \in \Irr(G)} \.
c(\rho,\pi) \. \frac{\pi(1)\cdot [G:H]}{\rho(1)} \, = \,  \frac{1}{\rho(1)} \ts \cdot \ts \rho(1)\ts [G:H]
\, = \,  [G:H]\ts,
$$
where we repeatedly use both equations in~\eqref{eq:LR-groups-def}.
\end{proof}

\begin{cor}  \label{c:LR-squared-groups}
For every $H < G$, we have:
$$
[G:H] \, \le \,
\sum_{\rho \in \Irr(G)} \. \sum_{\pi \in \Irr(H)} \. c(\rho,\pi)^2
\, \le \, [G:H] \. \min\bigl\{\vk(G),\ts\vk(H)\bigr\} \ts.
$$
\end{cor}

Note that $\vk(H)$ can be much larger that $\vk(G)$.  For example,
take \ts $H=\zz_2^{n/2}$ and $G=S_n$.  Then $\vk(H)=2^{n/2}$, while
$\vk(S_n) = e^{\Theta(\sqrt{n})}$.

\medskip

\subsection{Largest induced multiplicity}\label{s:LR-max}
Recall the definition of $\rLR(G,H)$ from the introduction.  We have:

\begin{proof}[Proof of Theorem~\ref{t:LR-intro-bounds}]
The lower bound follows immediately from Corollary~\ref{c:LR-squared-groups-ineq},
while the upper bound follows from Lemma~\ref{l:LR-squared-groups-2}.
\end{proof}

\begin{proof}[Proof of Theorem~\ref{t:LR-intro-exist}]
Let $\pi$ be the character in the largest term in the RHS of
$$
|G|^{1/2}/a \, \le \, \rho(1) \, = \, \sum_{\pi\in \Irr(H)} \. c(\rho,\pi)\. \pi(1)\ts.
$$
On the one hand, by the upper bound in Theorem~\ref{t:LR-intro-bounds} we have:
$$
\pi(1) \, \ge \, \frac{\rho(1)}{\vk(H)\cdot \rLR(G,H)}  \, \ge \,
\frac{|G|^{1/2}/a}{\vk(H)\cdot [G:H]^{1/2}} \, = \,
\frac{|H|^{1/2}}{a \ts \vk(H)}\..
$$
On the other hand,
$$
c(\rho,\pi) \, \ge \, \frac{\rho(1)}{\vk(H)\cdot \rD(H)}
\, \ge \,
\frac{|G|^{1/2}/a}{\vk(H)\cdot |H|^{1/2}}
\, = \,
\frac{[G:H]^{1/2}}{a \ts \vk(H)}\,,
$$
as desired.
\end{proof}

\begin{rem} \label{r:LR-groups}
{\rm Theorem~\ref{t:LR-intro-exist} above is patterned after
Theorem~\ref{t:Kron-main-intro}.
Note, however, that we do not have an analogue of a much simpler
Proposition~\ref{p:Kron-main-groups}.
}\end{rem}

\bigskip

\section{Final remarks}\label{s:fin-rem}

\subsection{}\label{ss:finrem-hist}
The study of $\rD(S_n)$ was initiated back in 1954, in one of the
earliest uses of computer calculations in combinatorics and algebra~\cite{BMSW}.
The study continued in a long series of papers~\cite{BDJ,Mc,LS,VK1,VK2,VP},
in part due to connections to random unitary matrices and longest increasing
subsequences in random permutations.  Let us single out
papers~\cite{LS,VK1} which determined the \emph{limit shape} of the
partition~$\la$ corresponding to the largest character~$\chi^\la$,
and~\cite{BDJ} which determined the exact distribution of shapes~$\la$.
Stanley's questions are best viewed as part of this research direction.
We refer to~\cite{Rom} for a comprehensive overview of the area.

For general groups, parameter Snyder in~\cite{Sny} introduced parameter
$e(G)$ defined by $\rD(G) \bigl(\rD(G)+e(G)\bigr)= |G|$.  Parameter
$e(G)$  is closely related to~$\ve(G)$, and was the motivation for
a series of recent papers improving bound on both~\cite{HLS,HHN,Isa,LMT}.

\subsection{}\label{ss:finrem-Stanley}
The literature on Kronecker and Littlewood--Richardson coefficients is so vast, there is no
single source that would give it justice.  We refer to~\cite{Stanley-EC2}
for a comprehensive introduction to the subject and to~\cite{Ful,vL} for
connections to Algebra and Geometry, and to~\cite{PPY} for further
references.

We should mention that from the point of view of Schur duality,
one can describe the classical Littlewood--Richardson coefficients $c^\la_{\mu\nu}$ of $S_n$
as a special case of Kronecker multiplicities for $\GL_N(q)$.  Indeed, by
taking \ts $N \ge 2\ell(\la)$ and $q$ large enough, the
Kronecker multiplicities \ts $g(\chi^\la,\chi^\mu,\chi^\nu)$ \ts
of the corresponding $\GL_N(q)$-reps become polynomial in~$q$.
Letting $q\to 1$ in these polynomials recovers $c^\la_{\mu\nu}$.
Thus, estimating the Kronecker multiplicities for $\GL_N(q)$
is likely to be difficult.

\subsection{}\label{ss:finrem-buf}
It was shown by Bufetov~\cite{Buf} that w.r.t.\ the
Plancherel measure there is a concentration of \ts
$$\frac{1}{\sqrt{n}} \log \frac{\chi^\la(1)^2}{n!} \quad \text{as} \ \, n\to \infty
$$
at some $h \in [-2\co,-2\ct]$, where $c_1,c_2$ are given in~\eqref{eq:VK-const}.
If such $h$ was determined, this would further
improve the asymptotic bounds on $\ve(S_n)$ given in the proof of Theorem~\ref{t:Sn-ve}.
Numerical experiments in~\cite{VP} suggest that there is a limit
$$
\eta \, = \, \lim_{n\to \infty} \. \frac{1}{\sqrt{n}} \log \frac{\rD(S_n)^2}{n!}
$$
and that $h<\eta$.

\subsection{}\label{ss:finrem-sum-dim}
It was noted by McKay~\cite{Mc} and Kowalski~\cite[p.~80]{Kow} that
for some families of groups a nice interpretation for the sum of degrees
are known:
$$
f(G) \, = \, \sum_{\chi\in \Irr(G)} \. \chi(1)\ts.
$$
Namely, $f(S_n)$ is the number of involutions, $f\bigl(\GL_n(q)\bigr)$
is the number of symmetric matrices, etc.  We refer to~\cite{Vin} for
the unified view of these results and review of prior work by Gow,
Klyachko, and others.  We should mention that for our applications,
these formulas give weaker bounds compared to~\eqref{eq:dim-groups}.
For~$S_n$, this was pointed out in~\cite{VK2}, who improved
upon McKay's lower bound.

\subsection{}\label{ss:finrem-LR} It would be interesting to see
if we always have \ts $\rLR(G,H)\le \sqrt{b(G)/b(H)}$,
which would be sharper in some cases and match the upper bound \ts
$\rK(G)\le \rD(G)$ \ts in the diagonal embedding case, see
Remark~\ref{r:intro-diag}.  We have not checked this speculation
on a computer.

\vskip.76cm

\subsection*{Acknowledgements}
We are grateful to Jan Saxl and Richard Stanley for introducing us
to the area, and to Persi Diaconis, Olivier Dudas, Alejandro Morales,
Leonid Petrov, Dan Romik, Rapha\"el Rouquier and Pham Tiep for
helpful comments and discussions.  Alexander Hulpke kindly computed
the value $\rK(M)$ in~$\S$\ref{ss:monster}.  The first and second
authors were partially supported by the~NSF.

\vskip.8cm

\end{document}